\documentclass[11pt]{article}

\usepackage[top=50mm, bottom=50mm, left=50mm, right=50mm]{geometry}
\usepackage{lineno}
\usepackage{amssymb}
\usepackage{amsmath}
\usepackage{amsthm}
\usepackage{amsfonts}
\usepackage{epsfig}
\usepackage{graphicx}
\usepackage{graphics}
\usepackage{float}
\usepackage{multirow}
\usepackage{lineno}
\usepackage{fullpage}
\usepackage[normalem]{ulem} 
\usepackage{makeidx}
\usepackage{xspace}
\usepackage{wrapfig}
\usepackage{color}

\makeindex

\theoremstyle{definition}
\newtheorem{theorem}{Theorem}
\newtheorem{defi}[theorem]{Definition}

\newtheorem*{rmk}{Remark}

\newtheorem{lem}[theorem]{Lemma}

\numberwithin{theorem}{section}  
\numberwithin{equation}{section} 
\newtheorem*{asm}{Assumption}

\newcommand{\Pro}{{\mathbf{P}}}

\newcommand{\Pas}{{\Pro\text{-a.s.}}}

\allowdisplaybreaks[4]

\begin{document}

\title{Global solution for the stochastic nonlinear Schr$\ddot{\text{o}}$dinger system with quadratic interaction in four dimensions\footnote{ASM Subject Classifications: 60H15, 35B65, 35J10 \\ keywords: Stochastic nonlinear Schr\"odinger equation; Ground state}}
 
\author{
{Masaru Hamano\footnote{Faculty of Science and Engineering, Waseda University, Tokyo 169-8555, Japan, email: m.hamano3@kurenai.waseda.jp}} \and {Shunya Hashimoto\footnote{Department of Mathematics, Faculty of Science, Saitama University, Saitama 338-8570, Japan, email: s.hashimoto.230@ms.saitama-u.ac.jp}}  \and {Shuji Machihara\footnote{Department of Mathematics, Faculty of Science, Saitama University, Saitama 338-8570, Japan, email: machihara@rimath.saitama-u.ac.jp}}
}

\date{}

\maketitle

\begin{abstract}
We discuss the global existence of solutions to a system of stochastic Schr\"odinger equations with multiplicative noise.
Our setting of the quadratic nonlinear terms in dimension 4 is $L^2$-critical.
We treat the solutions under the ground state.
We estimate the time derivative of the quantity of energy by using the cancellation of the cubic terms in the spatial derivative of the solution.
\end{abstract}

\section{Introduction}
We consider the Cauchy problem for the stochastic nonlinear Schr\"odinger system (SNLSS) with multiplicative noise:
\begin{eqnarray}
\label{SNLSS}
\begin{cases}
du(t,\xi)=i\Delta u(t,\xi)dt+2iv(t,\xi)\overline{u(t,\xi)}dt \\
\hspace{5em} -\mu(\xi)u(t,\xi)dt+u(t,\xi)dW(t,\xi), \quad t\in(0,T), \ \xi\in \mathbb{R}^d, \\
dv(t,\xi)=\frac{1}{2}i\Delta v(t,\xi)dt+iu^2(t,\xi)dt \\
\hspace{5em} -\widetilde{\mu}(\xi)v(t,\xi)dt+v(t,\xi)d\widetilde{W}(t,\xi), \quad t\in(0,T), \ \xi\in \mathbb{R}^d, \\
u(0,\xi)=u_0(\xi), \quad v(0,\xi)=v_0(\xi), \quad \xi\in\mathbb{R}^d.
\end{cases}
\end{eqnarray}
The Wiener processes $W(t,\xi), \widetilde{W}(t,\xi)$ and the functions $\mu, \widetilde{\mu}$ are given by
\begin{align*}
W(t,\xi)=&\sum_{k=1}^Ni\phi_k(\xi)B_k(t), \quad \widetilde{W}(t,\xi)=2W(t,\xi), \\
&\mu=\frac{1}{2}\sum_{k=1}^N\phi_k^2, \quad \widetilde{\mu}=4\mu,
\end{align*}
where, $\phi_k\in C^{\infty}_b(\mathbb{R}^d,\mathbb{R})$ and the $B_k(t)$ are real-valued independent Brownian motions with respect to a probability space $(\Omega,\mathcal{F},\Pro)$ with natural filtration $(\mathcal{F}_t)_{t\ge0}, \ 1\le k\le N$.
In this paper, we assume $N<\infty$ which is the same setting with the papers \cite{BRZ14,BRZ16}. Our techniques easily go over to the case where $N=+\infty$ (i.e. infinite dimensional noise), we refer to the paper \cite[Remark 2.3.13]{Z14} for this extension. 

In the deterministic case (i.e., $W=0$), \eqref{SNLSS} is the following equation.
\begin{eqnarray}
\label{NLSS}
\begin{cases}
du(t,\xi)=i\Delta u(t,\xi)dt+2iv(t,\xi)\overline{u(t,\xi)}dt \quad t\in(0,T), \ \xi\in \mathbb{R}^d, \\
dv(t,\xi)=\frac{1}{2}i\Delta v(t,\xi)dt+iu^2(t,\xi)dt \quad t\in(0,T), \ \xi\in \mathbb{R}^d, \\
u(0,\xi)=u_0(\xi), \quad v(0,\xi)=v_0(\xi), \quad \xi\in\mathbb{R}^d.
\end{cases}
\end{eqnarray}
We introduce the conservation laws related to \eqref{NLSS} as
\begin{align}
\label{mass}
M(u,v)(t)&:=\|u\|^2_{L^2}+2\|v\|^2_{L^2}=M(u_0,v_0), \\
\label{enecon}
E(u,v)(t)&:=K(u,v)-2P(u,v)=E(u_0.v_0),
\end{align}
where
\begin{align}
\label{K}
K(u,v)(t)&:=\|\nabla u\|^2_{L^2}+\frac{1}{2}\|\nabla v\|^2_{L^2}, \\
\label{P}
P(u,v)(t)&:=\text{Re}\langle v,u^2\rangle,
\end{align}
with
\[ \langle f,g\rangle:=\int_{\mathbb{R}^d}f(x)\overline{g(x)}dx. \]
N. Hayashi, T. Ozawa, and K. Tanaka showed the local well-posedness of \eqref{NLSS} in \cite{HOT13}. They also showed that if the mass of the initial data is less than the mass of the ground state $(\phi,\psi)$, then the solution exists globally, where $(\phi,\psi)\in H^1\times H^1$ is a nontrivial solution of the elliptic equation
\begin{align}
\label{ground}
\begin{cases}
-\Delta \phi+\phi=2\psi\phi, \\
-\frac{1}{2}\Delta \psi+2\psi=\phi^2,
\end{cases}
\end{align}
and is given as attaining the infimum of 
\begin{align*}
I(\phi,\psi)=\frac{1}{2}M(\phi,\psi)+\frac{1}{2}E(\phi,\psi),
\end{align*}
with a critical point of $I$.

Next, we introduce the results for the stochastic case. For the single stochastic Schr\"odinger equation, global and blow-up solutions have been studied in \cite{BRZ14,BRZ16,BRZ17,BD99,BD03}. In particular, \cite{BRZ14,BRZ16,BRZ17} is used a rescaling transformation introduced in \cite{BDR09} under the additional assumptions on the noise coefficients.
This transformation reduces the stochastic Schr\"odinger equation to a random Schr\"odinger equation which is independent of noise. 

Also, in the case of stochastic Schr\"odinger systems, Y. Chen, J. Duan, and Q. Zhang \cite{CDZ20} studied the stochastic nonlinear Schr\"odinger system with usual power-type nonlinear terms with exponent $2\sigma+1$, namely, 
\begin{align*}
\begin{cases}
idu+(\Delta u+(\lambda_{11}|u|^{2\sigma}+\lambda_{12}|v|^{\sigma+1}|u|^{\sigma-1})u)dt=u\circ \phi_1dW(t), \\
idv+(\Delta v+(\lambda_{21}|v|^{\sigma-1}|u|^{\sigma+1}+\lambda_{22}|u|^{2\sigma})v)dt=v\circ \phi_2dW(t), \\
u(0,\xi)=u_0(\xi), \quad v(0,\xi)=v_0(\xi),
\end{cases}
\end{align*}
where the coefficients $\lambda_{ij}\in \mathbb{R}$ for $i,j=1,2, \ (W(t))_{t\ge0}$ is a cylindrical Wiener process in the conservative case, the notation $\circ$ stands for Stratonovich product in the right-hand side, and $\phi_1, \phi_2$ are Hilbert-Schmidt operators from $L^2(\mathbb{R}^d)$ into $H^1(\mathbb{R}^d)$.
In \cite{CDZ20}, using the method by A. de Bouard and A. Debussche \cite{BD99,BD03}, they show the $H^1$- local well-posedness for $\sigma\in [0,\frac{2}{d})\cup (\frac{1}{2},\frac{2}{(d-2)^+})$ and the $H^1$- global well-posedness for $\sigma\in [0,\frac{2}{d}]$ (with an additional assumption when $\sigma=\frac{2}{d}$). 

For \eqref{SNLSS}, the authors' previous work in \cite{HHM23} shows $L^2$-local well-posedness in dimension $1\le d\le 4$, $L^2$-global well-posedness in dimension $1\le d\le 3$, $H^1$-local well-posedness in dimension $1\le d\le 6$, and $H^1$-global well-posedness in dimension $1\le d\le 3$. 
So this current paper is a natural continuation of these studies on $H^1$-global result in dimension 4.

In this paper, the existence of $H^1$-global solutions at $d=4$ using rescaling transformations is shown below the mass of the ground state $(\phi,\psi)$ (Theorem \ref{uvglo4}).
We remark that dimension 4 is $L^2$-critical dimension for the quadratic nonlinearity.

\section{Results}

In this section, we introduce the rescaling transformation and the main results.
\begin{defi}
\label{uvdef}
Let $u_0,v_0$ belong to $H^1$. Fix $0<T<\infty$. We say that a triple $(u,v,\tau)$ is a $H^1$-solution of \textnormal{(\ref{SNLSS})}, where $\tau(\le T)$ is an $(\mathcal{F}_t)$-stopping time, and $u=(u(t))_{t\in [0,\tau]},v=(v(t))_{t\in [0,\tau]}$ is an $H^1$-valued continuous $(\mathcal{F}_t)$-adapted process, such that $v\overline{u},u^2\in L^1(0,\tau;H^{-1})$, $\Pas$, and it satisfies $\Pas$
\begin{align}
u(t)&=u_0+\int_0^t(i\Delta u(s)-\mu u(s)+2iv(s)\overline{u(s)})ds+\int_0^tu(s)dW(s), \quad t\in[0,\tau], \\
v(t)&=v_0+\int_0^t(i\frac{1}{2}\Delta v(s)-\overline{\mu} v(s)+iu^2(s))ds+\int_0^tv(s)d\widetilde{W}(s), \quad t\in[0,\tau],
\end{align}
as a system of equations in $H^{-1}$.
\end{defi}
We assume the following decay conditions for the noise coefficients $(\phi_k)_{1\le k\le N}$.
\begin{asm}
\textbf{(H)} \  (Asymptotical flatness) For any $1\le k\le N$, $\phi_k$ satisfies the following for any multi-index $\nu$.
\begin{align}
\lim_{|\xi|\to\infty}\langle \xi\rangle^2|\partial^{\nu}_{\xi}\phi_k(\xi)|=0,
\end{align}
where the bracket is defined by
\[ \langle \xi \rangle:=\sqrt{1+|\xi|^2}. \]
\end{asm}
We introduce the results of local well-posedness and the blow-up alternative for \eqref{SNLSS}.
The authors have shown the local well-posedness for $1\le d\le 6$ in \cite{HHM23}. Here, we state it, especially in $d=4$.
\begin{theorem}
\label{mainH}
(The local well-posedness for \eqref{SNLSS} in four dimensions \cite{HHM23})
Let $d=4$. Assume (H). Then, for each $u_0,v_0\in H^1$ and $0<T<\infty$, there exist a $H^1$-local solution $(u,v,\tau^*)$ of \textnormal{(\ref{SNLSS})}, where $\tau^*\in (0,T]$ is a stopping times. For any $t<\tau^*$, it holds $\Pas$ that 
\begin{align}
u|_{[0,t]},v|_{[0,t]}\in C([0,t];H^1)\cap L^2(0,t;W^{1,4}).
\end{align}
Moreover, we have the blowup alternative, that is, for $\Pas \ \omega$, either $\tau^*(u_0,v_0)(\omega)=T$ or
\begin{align}
\lim_{t\to \tau^*(u_0,v_0)(\omega)}(\|u(t)(\omega)\|_{H^1}&+\|v(t)(\omega)\|_{H^1})=\infty.
\end{align}
\end{theorem}
The proof is based on the equivalence of the solutions of the equations (\ref{SNLSS}) and that of (\ref{RSNLSS}).

To connect the stochastic system with white noise (\ref{SNLSS}) into the following random system without white noise, we consider the rescaling transformation $u=e^Wy, \ v=e^{\widetilde{W}}z$.
Then, $y,z$ satisfy
\begin{eqnarray}
\label{RSNLSS}
\begin{cases}
\displaystyle \frac{\partial y(t,\xi)}{\partial t}=A_1(t)y(t,\xi)+2iz(t,\xi)\overline{y(t,\xi)}, \\
\displaystyle \frac{\partial z(t,\xi)}{\partial t}=A_2(t)z(t,\xi)+iy^2(t,\xi), \\
y(0,\xi)=u_0(\xi), \quad z(0,\xi)=v_0(\xi).
\end{cases}
\end{eqnarray}
Here,
\begin{align}
\label{EO}
A_1(t)y(t,\xi)&=ie^{-W(t,\xi)}\Delta (e^{W(t,\xi)}y(t,\xi)) \nonumber \\
&=i(\Delta +b_1(t,\xi)\cdot \nabla +c_1(t,\xi))y(t,\xi), \\
\label{EOL}
A_2(t)z(t,\xi)&=\frac{1}{2}ie^{-\widetilde{W}(t,\xi)}\Delta (e^{\widetilde{W}(t,\xi)}z(t,\xi)) \nonumber \\
&=i(\frac{1}{2}\Delta +b_2(t,\xi)\cdot \nabla +c_2(t,\xi))z(t,\xi),
\end{align}
with
\begin{align}
\label{b}
b_1(t,\xi)&=2\nabla W(t,\xi)=2i\sum_{k=1}^N\nabla \phi_k(\xi)B_k(t), \\
b_2(t,\xi)&=\nabla \widetilde{W}(t,\xi)=i\sum_{k=1}^N\nabla (2\phi_k(\xi))B_k(t), \\
\label{c}
c_1(t,\xi)&=\sum_{j=1}^4(\partial_jW(t,\xi))^2+\Delta W(t,\xi) \nonumber \\
&=-\sum_{j=1}^4(\sum_{k=1}^N\partial_j\phi_k(\xi)B_k(t))^2+i\sum_{k=1}^N\Delta \phi_k(\xi)B_k(t), \\
c_2(t,\xi)&=\frac{1}{2}\sum_{j=1}^4(\partial_j\widetilde{W}(t,\xi))^2+\frac{1}{2}\Delta \widetilde{W}(t,\xi) \nonumber \\
&=-\frac{1}{2}\sum_{j=1}^4(\sum_{k=1}^N\partial_j(2\phi_k(\xi))B_k(t))^2+\frac{1}{2}i\sum_{k=1}^N\Delta (2\phi_k(\xi))B_k(t).
\end{align}
We give the definition of $H^1$-solution for the rescaled system \eqref{RSNLSS}.
\begin{defi}
\label{yzdef}
Let $u_0,v_0$ belong to $H^1$. Fix $0<T<\infty$. We say that a triple $(y,z,\tau)$ is a $H^1$-solution of (\ref{RSNLSS}), where $\tau(\le T)$ is an $(\mathcal{F}_t)$-stopping time, and $y=(y(t))_{t\in [0,\tau]},z=(z(t))_{t\in [0,\tau]}$ is an $H^1$-valued continuous $(\mathcal{F}_t)$-adapted process, such that $z\overline{y},y^2\in L^1(0,\tau;H^{-1})$, $\Pas$, and it satisfies $\Pas$
\begin{align}
y(t)&=u_0+\int_0^tA_1(s)y(s)ds+\int_0^t2iz(s)\overline{y(s)}ds, \quad t\in[0,\tau], \\
z(t)&=v_0+\int_0^tA_2(s)z(s)ds+\int_0^tiy^2(s)ds, \quad t\in[0,\tau].
\end{align}
as a system of equations in $H^{-1}$.
\end{defi}
We introduce the equivalence of the solution of the stochastic equation (\ref{SNLSS}) and the solution of the random equation (\ref{RSNLSS}).
\begin{theorem}
\label{equiv}
(The equivalence of the solutions for the systems of (\ref{SNLSS}) and (\ref{RSNLSS}), \cite{HHM23})
\begin{enumerate}
\item Let $(y,z,\tau)$ be a $H^1$-solution of (\ref{RSNLSS}) in the sense of Definition \ref{yzdef}. Set $u:=e^Wy,v:=e^{\widetilde{W}}z$. Then $(u,v,\tau)$ is a $H^1$-solution of (\ref{SNLSS}) in the sense of Definition \ref{uvdef}.
\item Let $(u,v,\tau)$ be a $H^1$-solution of (\ref{SNLSS}) in the sense of Definition \ref{uvdef}. 
Set $y:=e^{-W}u,z:=e^{-\widetilde{W}}v$. Then $(y,z,\tau)$ is a $H^1$-solution of (\ref{RSNLSS}) in the sense of Definition \ref{yzdef}.
\end{enumerate}
\end{theorem}
By the equivalence of two expressions of solutions via the rescaling transformations (\ref{SNLSS}) and (\ref{RSNLSS}),
Theorem \ref{mainH} is rewritten as Theorem \ref{ymain}.
Here, we state it, especially in $d=4$.
\begin{theorem}
\label{ymain}
(The local well-posedness for \eqref{RSNLSS} in four dimensions \cite{HHM23})
Let $d=4$. Assume (H). Then, for each $u_0,v_0\in H^1$ and $0<T<\infty$, there exist a $H^1$-local solution $(y,z,\tau^*)$ of (\ref{RSNLSS}), where $\tau^*\in (0,T]$ is a stopping times. For any $t<\tau^*$, it holds $\Pas$ that 
\begin{align}
y|_{[0,t]},z|_{[0,t]}\in C([0,t];H^1)\cap L^2(0,t;W^{1,4}).
\end{align}
Moreover, we have the blowup alternative, that is, for $\Pas \ \omega$, either $\tau^*(u_0,v_0)(\omega)=T$, or
\begin{align}
\lim_{t\to \tau^*(u_0,v_0)(\omega)}(\|y(t)(\omega)\|_{H^1}&+\|z(t)(\omega)\|_{H^1})=\infty.
\end{align}
\end{theorem}

The following result is the main theorem in this paper.
\begin{theorem}
\label{uvglo4}
(Global well-posedness for (\ref{SNLSS}) below the ground state)
Let $d=4$ and $(\phi,\psi)$ be a ground state to \eqref{ground}. Assume (H). If $u_0,v_0\in H^1$ satisfy $M(u_0,v_0)<M(\phi,\psi)$, then the corresponding solution $(u,v)$ to (\ref{SNLSS}) exists globally almost surely.
\end{theorem}
By the equivalence of two expressions of solutions via the rescaling transformations (\ref{SNLSS}) and (\ref{RSNLSS}),
Theorem \ref{uvglo4} is rewritten as Theorem \ref{yzglo4}.
\begin{theorem}
\label{yzglo4}
(Global well-posedness for (\ref{RSNLSS}) below the ground state)
Let $d=4$ and $(\phi,\psi)$ be a ground state to \eqref{ground}. Assume (H). If $u_0,v_0\in H^1$ satisfy $M(u_0,v_0)<M(\phi,\psi)$, then the corresponding solution $(y,z)$ to (\ref{RSNLSS}) exists globally almost surely.
\end{theorem}
We prove Theorem \ref{yzglo4} by using the argument in \cite{HOT13}. We use the upper bound of energy (Lemma \ref{upper}).
The difficulty in the proof is that the system of equations \eqref{RSNLSS} after rescaling transformations of \eqref{SNLSS} and also \eqref{SNLSS} have conservation of mass \eqref{mass} but not conservation of energy \eqref{enecon}.
This makes it more difficult when considering the time derivative of the energy than the single case; since we estimate the integral of cubic terms $(\nabla \overline{y})z\overline{y}, \ (\nabla \overline{z})y^2$, and $z\overline{y}^2$, specific to the \eqref{RSNLSS} here, we recall the single Schr\"odinger equation no cubic terms appear in estimating the energy \cite{SZ23}. 
It seems to be obvious that cubic terms including spatial derivatives are more difficult than other cubic terms without spatial derivatives, and indeed it is.
However, it was found that a long and appropriate calculation given as the integral terms in $(\nabla \overline{y})z\overline{y}$, and $(\nabla \overline{z})y^2$ can be canceled.
Therefore only the integral term of $z\overline{y}^2$ needs to be estimated, and we have succeeded it in showing the existence of $H^1$-global solutions even if the energy is not conserved.

In the next section, the proof of Theorem \ref{yzglo4} will be given, which is sufficient to obtain Theorem \ref{uvglo4}.

\section{Proof of Main theorem.}

In this section, we prove Theorem \ref{yzglo4} to see Theorem \ref{uvglo4} as well.
We recall the Gagliardo-Nirenberg inequality, which is shown in \cite{HOT13}.
\begin{lem}
\label{GNI}
(Gagliardo-Nirenberg inequality \cite[Theorem 5.1]{HOT13})
Let $f,g\in H^1$ and $(\phi,\psi)$ be a ground state to \eqref{ground}. Then, we have
\[ P(f,g)\le \frac{1}{2}\sqrt{\frac{M(f,g)}{M(\phi,\psi)}}K(f,g), \]
where $K$ and $P$ are given in \eqref{K} and \eqref{P}, respectively.
\end{lem}
\begin{lem}
\label{Com}
(Conservation of mass)
The solution $(y,z)$ to \eqref{RSNLSS} conserves its mass. That is,
\[ M(y(t),z(t))=M(u_0,v_0). \]
\end{lem}
\begin{rmk}
Now we consider the conservative case (i.e. $\text{Re} W=0$), so mass is conserved \cite[Lemma 4.1]{HHM23}.
\end{rmk}
Next, we estimate energy.
\begin{lem}
\label{upper}
(Upper bound of energy)
The solution $(y,z)$ to \eqref{RSNLSS} on $[0,\tau^*)$ satisfies
\[ E(y(t),z(t))\le E(u_0,v_0)+C(\tau^*)+C(\tau^*)\int_0^tK(y(s),z(s))ds, \]
for any $0<t<\tau^*$.
\end{lem}
\begin{proof}
We differentiate the energy with time, we have
\begin{align*}
&\frac{d}{dt}E(y,z) \\
&=\frac{d}{dt}\langle \nabla y,\nabla y \rangle+\frac{1}{2}\frac{d}{dt}\langle \nabla z,\nabla z \rangle-2\frac{d}{dt}\text{Re}\langle z,y^2\rangle \nonumber \\
&=-2\text{Re}\langle \Delta y,\partial_t y\rangle-2\text{Re}\langle \frac{1}{2}\Delta z,\partial_tz \rangle-2\text{Re}\langle y^2,\partial_tz\rangle-2\text{Re}\langle 2z\overline{y},\partial_ty\rangle \\
&=-2\text{Re}\langle \Delta y+2z\overline{y},\partial_t y\rangle-2\text{Re}\langle \frac{1}{2}\Delta z+y^2,\partial_tz \rangle \\
&=-2\text{Im}\langle \Delta y+2z\overline{y},-i\partial_ty\rangle -2\text{Im}\langle \frac{1}{2}\Delta z+y^2,-i\partial_tz\rangle \nonumber \\
&=-2\text{Im}\langle \Delta y+2z\overline{y},b_1\cdot \nabla y+c_1y\rangle-2\text{Im}\langle \frac{1}{2}\Delta z+y^2,b_2\cdot \nabla z+c_2z\rangle \\
&=-2\text{Im}\int \overline{b_1\cdot \nabla y}\Delta y d\xi-2\text{Im}\int \overline{c_1y}\Delta yd\xi-2\text{Im}\int \overline{b_1\cdot \nabla y}(2z\overline{y}) d\xi-2\text{Im}\int \overline{c_1y}(2z\overline{y})d\xi \\
& \quad -2\text{Im}\int \overline{b_2\cdot \nabla z}\left(\frac{1}{2}\Delta z\right) d\xi-2\text{Im}\int \overline{c_2z}\left(\frac{1}{2}\Delta z\right)d\xi-2\text{Im}\int \overline{b_2\cdot \nabla z}y^2 d\xi-2\text{Im}\int \overline{c_2z}y^2d\xi \\
&=-2\text{Im}\int (2\nabla \overline{W}\cdot \nabla \overline{y})\Delta yd\xi-2\text{Im}\int (2\nabla \overline{W}\cdot\nabla \overline{y})(2z\overline{y})d\xi \\
& \quad -2\text{Im}\int\left(\overline{\sum_{j=1}^4(\partial_jW)^2+\Delta W}\right)\overline{y}\Delta yd\xi-2\text{Im}\int\left(\overline{\sum_{j=1}^4(\partial_jW)^2+\Delta W}\right)\overline{y}(2z\overline{y})d\xi \\
& \quad -2\text{Im}\int (\nabla \overline{\widetilde{W}}\cdot \nabla \overline{z})\left(\frac{1}{2}\Delta z\right)d\xi-2\text{Im}\int (\nabla \overline{\widetilde{W}}\cdot\nabla \overline{z})y^2d\xi \\
& \quad -2\text{Im}\int\left(\overline{\frac{1}{2}\sum_{j=1}^4(\partial_j\widetilde{W})^2+\frac{1}{2}\Delta \widetilde{W}}\right)\overline{z}\left(\frac{1}{2}\Delta z\right)d\xi-2\text{Im}\int\left(\overline{\frac{1}{2}\sum_{j=1}^4(\partial_j\widetilde{W})^2+\frac{1}{2}\Delta \widetilde{W}}\right)\overline{z}y^2d\xi \\
&=:\sum_{j=1}^8I_j.
\end{align*}
For the integral terms of $(\nabla \overline{y})z\overline{y}$, and $(\nabla \overline{z})y^2$ appearing in the second and sixth terms of the final formula, there is not enough regularity to estimate those in the $L^2$ norm.
However, the terms of the cubic polynomial which includes $\nabla$, which is an obstacle to the proof, can be canceled by appropriate calculations.

Indeed, first, we compute from $I_1$ to $I_4$ and we obtain
\begin{align*}
I_1&=-2\text{Im}\int (2\nabla \overline{W}\cdot \nabla \overline{y})\Delta yd\xi=4\sum_{k=1}^NB_k\text{Re}\int \nabla\phi_k\cdot \nabla \overline{y}\Delta yd\xi, \\
I_2&=-2\text{Im}\int (2\nabla \overline{W}\cdot\nabla \overline{y})(2z\overline{y})d\xi=-4\sum_{k=1}^N\text{Im}\int (\nabla(\overline{i\phi_kB_k})\cdot \nabla \overline{y})(2z\overline{y})d\xi \nonumber \\
&=4\sum_{k=1}^NB_k\text{Re}\int(\nabla \phi_k\cdot\nabla \overline{y})(2z\overline{y})d\xi, \\
I_3&=-2\text{Im}\int\left(\overline{\sum_{j=1}^4(\partial_jW)^2+\Delta W}\right)\overline{y}\Delta yd\xi \nonumber \\
&=-2\sum_{j=1}^4\text{Im}\int (\partial_jW)^2\overline{y}\Delta yd\xi-2\text{Im}\int (\Delta \overline{W})\overline{y}\Delta yd\xi \nonumber \\
&=2\sum_{j=1}^4\text{Im}\int \left( \sum_{k=1}^N\partial_j\phi_kB_k \right)^2 \overline{y}\Delta y d\xi+2\sum_{k=1}^NB_k\text{Re}\int \Delta \phi_k\overline{y} \Delta y d\xi \nonumber \\
&=-2\sum_{j=1}^4\text{Im}\int \nabla \left( \sum_{k=1}^N\partial_j\phi_kB_k \right)^2\cdot \nabla y\overline{y}d\xi+2\sum_{k=1}^NB_k\text{Re}\int\Delta \phi_k\overline{y}\Delta yd\xi, \\
I_4&=-2\text{Im}\int\left(\overline{\sum_{j=1}^4(\partial_jW)^2+\Delta W}\right)\overline{y}(2z\overline{y})d\xi \nonumber \\
&=-2\sum_{j=1}^4\text{Im}\int (\partial_jW)^2\overline{y}(2z\overline{y})d\xi-2\text{Im}\int \Delta \overline{W}\overline{y}(2z\overline{y})d\xi \nonumber \\
&=2\sum_{j=1}^4\text{Im}\int \left( \sum_{k=1}^N\partial_j\phi_kB_k \right)^2\overline{y}(2z\overline{y})d\xi+2\sum_{k=1}^NB_k\text{Re}\int\Delta \phi_k \overline{y}(2z\overline{y})d\xi.
\end{align*}
Therefore, we have
\begin{align}
I_1+I_3&=4\sum_{k=1}^NB_k\text{Re}\int \nabla\phi_k\cdot \nabla \overline{y}\Delta yd\xi \nonumber \\
&\quad -2\sum_{j=1}^4\text{Im}\int \nabla \left( \sum_{k=1}^N\partial_j\phi_kB_k \right)^2\cdot \nabla y\overline{y}d\xi+2\sum_{k=1}^NB_k\text{Re}\int\Delta \phi_k\overline{y}\Delta yd\xi \nonumber \\
&=\sum_{k=1}^NB_k\int\Delta^2 \phi_k|y|^2d\xi-4\sum_{k=1}^NB_k\text{Re}\sum_{i=1}^4\sum_{j=1}^4\int \partial_i\partial_j\phi_k\partial_iy\partial_j\overline{y}d\xi \nonumber \\
&\quad -2\sum_{j=1}^4\text{Im}\int \nabla \left( \sum_{k=1}^N\partial_j\phi_kB_k \right)^2\cdot \nabla y\overline{y}d\xi, \\
I_2+I_4&=4\sum_{k=1}^NB_k\text{Re}\int(\nabla \phi_k\cdot\nabla \overline{y})(2z\overline{y})d\xi \nonumber \\
& \quad +2\sum_{j=1}^4\text{Im}\int \left( \sum_{k=1}^N\partial_j\phi_kB_k \right)^2\overline{y}(2z\overline{y})d\xi+2\sum_{k=1}^NB_k\text{Re}\int\Delta \phi_k \overline{y}(2z\overline{y})d\xi.
\end{align}
Similarly, we compute from $I_5$ to $I_8$ and we obtain
\begin{align}
I_5+I_7&=\frac{1}{2}\sum_{k=1}^NB_k\int\Delta^2 \phi_k|z|^2d\xi-2\sum_{k=1}^NB_k\text{Re}\sum_{i=1}^4\sum_{j=1}^4\int \partial_i\partial_j \phi_k \partial_iz\partial_j\overline{z} d\xi \nonumber \\
&\quad -\frac{1}{2}\sum_{j=1}^4\text{Im}\int \nabla \left( \sum_{k=1}^N\partial_j(2\phi_k)B_k \right)^2\cdot \nabla z\overline{z}d\xi, \\
I_6+I_8&=4\sum_{k=1}^NB_k\text{Re}\int(\nabla \phi_k\cdot\nabla \overline{z})y^2d\xi+\sum_{j=1}^4\text{Im}\int \left( \sum_{k=1}^N\partial_j(2\phi_k)B_k \right)^2\overline{z}y^2d\xi \nonumber \\
& \quad +2\sum_{k=1}^NB_k\text{Re}\int\Delta \phi_k \overline{z}y^2d\xi.
\end{align}
Thus, we get
\begin{align}
\label{ene}
\frac{d}{dt}E(y,z)&=\sum_{k=1}^NB_k\int \Delta^2\phi_k|y|^2d\xi+\frac{1}{2}\sum_{k=1}^NB_k\int \Delta^2\phi_k|z|^2d\xi \nonumber \\
&-4\sum_{k=1}^NB_k\text{Re}\sum_{i=1}^4\sum_{j=1}^4\int \partial_i\partial_j\phi_k\partial_iy\partial_j\overline{y}d\xi-2\sum_{k=1}^NB_k\text{Re}\sum_{i=1}^4\sum_{j=1}^4\int \partial_i\partial_j \phi_k \partial_iz\partial_j\overline{z} d\xi \nonumber \\
&-2\sum_{j=1}^4\text{Im}\int\nabla\left(\sum_{k=1}^N\partial_j\phi_kB_k\right)^2\cdot\nabla y\overline{y}d\xi-\frac{1}{2}\sum_{j=1}^4\text{Im}\int\nabla\left(\sum_{k=1}^N\partial_j(2\phi_k)B_k\right)^2\cdot\nabla z\overline{z}d\xi \nonumber \\
&+2\sum_{k=1}^NB_k\text{Re}\int \Delta \phi_kz\overline{y}^2d\xi.
\end{align}
Then, for any $t'\in (0,\tau^*)$, since $\phi_k\in C^{\infty}_b$ and $B_k\in C([0,t']), \Pas, \  1\le k\le N$, using H\"older's inequality and Young's inequality, we obtain that $\Pas$ for any $t\in[0,t']$,
\begin{align}
\label{Eyz}
&E(y(t),z(t)) \nonumber \\
\le &E(u_0,v_0)+C(\tau^*)\int_0^t(\|y(s)\|^2_{L^2}+2\|z(s)\|^2_{L^2}+\|\nabla y(s)\|^2_{L^2}+\|\nabla z(s)\|^2_{L^2})ds \nonumber \\
&+C(\tau^*)\int_0^t P(y,z) ds.
\end{align}
From Lemma \ref{GNI} and Lemma \ref{Com}, we have
\begin{align}
\label{inter}
P(y(s),z(s))\le \frac{1}{2}\sqrt{\frac{M(y(s),z(s))}{M(\phi,\psi)}}K(y(s),z(s))=\frac{1}{2}\sqrt{\frac{M(u_0,v_0)}{M(\phi,\psi)}}K(y(s),z(s)).
\end{align}
In the last equality, we absorb one order of power, by Lemma \ref{Com}, so the third-order polynomial on the left-hand side is estimated by the second-order polynomial on the right-hand side.
We remark here that the same estimate in \eqref{inter} doesn't work for the terms $(\nabla \overline{y})z\overline{y}$ and $(\nabla \overline{z})y^2$.

Therefore, from \eqref{Eyz}, Lemma \ref{Com} and $t\le \tau^*$, we obtain
\begin{align}
\label{E1}
E(y(t),z(t))&\le E(u_0,v_0)+C(\tau^*)\int_0^t(M(y(s),z(s))+K(y(s),z(s)))ds \nonumber \\
&=E(u_0,v_0)+C(\tau^*)M(u_0,v_0)t+C(\tau^*)\int_0^tK(y(s),z(s))ds \nonumber \\
&\le (E(u_0,v_0)+C(\tau^*))+C(\tau^*)\int_0^tK(y(s),z(s))ds.
\end{align}
\end{proof}
We are ready to prove Theorem \ref{yzglo4}.
\begin{proof}[Proof of Theorem \ref{yzglo4}]
In Theorem \ref{ymain}, we have that $\Pas$ there exists a unique solution $(y,z)$ to equation (\ref{RSNLSS}) on $[0,\tau^*)$, where $\tau^*\in(0,T]$ is some positive random variable. Hence, to show the existence of a global solution, we only need to prove that $\tau^*=T, \ \Pas$. 

By Lemma \ref{GNI} and Lemma \ref{Com}, we obtain
\begin{align}
\label{E2}
E(y(t),z(t))&=K(y(t),z(t))-2P(y(t),z(t)) \nonumber \\ 
&\ge K(y(t),z(t))+2\left( -\frac{1}{2}\sqrt{\frac{M(u_0,v_0)}{M(\phi,\psi)}} \right)K(y(t),z(t)) \nonumber \\ 
&=\left( 1-\sqrt{\frac{M(u_0,v_0)}{M(\phi,\psi)}}\right) K(y(t),z(t)).
\end{align}
Therefore, from $M(u_0,v_0)<M(\phi,\psi)$, if we summarize (\ref{E1}) and (\ref{E2}), we have
\[ K(y(t),z(t))\le (CE(u_0,v_0)+C(\tau^*))+C'(\tau^*)\int_0^tK(y(s),z(s))ds. \]
Using Gronwall's inequality, we obtain $\Pas$
\begin{align}
\sup_{0\le t\le t'}(\|\nabla y(t)\|^2_{L^2}+\|\nabla z(t)\|^2_{L^2})\le C(\tau^*)<\infty,
\end{align}
which yields the boundedness of 
\[ \sup_{0\le t<\tau^*}(\|\nabla y(t)\|^2_{L^2}+\|\nabla z(t)\|^2_{L^2})<\infty, \]
by letting $t'\to \tau^*$. Therefore, using the blow-up alternative result, we obtain $\tau^*=T, \ \Pas$.
\end{proof}

\section*{Acknowledgement}
This work was supported by JSPS KAKENHI No. JP22J00787 (for M.H.) and No. JP19H00644 (for S.M.).
\section*{Declarations}
\subsection*{Conflicts of interests}
The authors declare that there is no conflict of interest regarding the publication of this paper.
\subsection*{Data Availability Statements}
Data sharing is not applicable to this article as no datasets were generated or analyzed during the current study.

\end{document}